\documentclass{amsart}

\usepackage{amsmath, amsfonts, amsthm, amssymb, verbatim,graphicx, hyperref}

\newtheorem{thm}{Theorem}

\newtheorem{cor}[thm]{Corollary}

\theoremstyle{definition}

\newcommand{\R}{\mathbb R}
\newcommand{\Z}{\mathbb Z}
\newcommand{\pd}{\partial}
\DeclareMathOperator{\tr}{Tr}
\DeclareMathOperator{\id}{id}

\begin{document}

\title{Axioms for the Lefschetz number as a lattice valuation}
\author{P. Christopher Staecker}
\address{Dept. of Mathematics and Computer Science, Fairfield University, Fairfield CT, 06824, USA}
\email{cstaecker@fairfield.edu}
\maketitle

\begin{abstract}
We give new axioms for the Lefschetz number based on Hadwiger's characterization of the Euler characteristic as the unique lattice valuation on polyhedra which takes value 1 on simplices. In the setting of maps on abstract simplicial complexes, we show that the Lefschetz number is unique with respect to a valuation axiom and an axiom specifying the value on a simplex. These axioms lead naturally to the classical computation of the Lefschetz number as a trace in homology. We then extend this approach to continuous maps of polyhedra, assuming an extra homotopy invariance axiom. We also show that this homotopy axiom can be weakened.
\end{abstract}

\section{Introduction}
Several papers in recent years have focused on characterizing the Lefschetz number and related invariants as the unique real valued function satisfying a set of natural axioms. Two notable recent approaches are by Arkowitz and Brown \cite{ab04}, and Furi, Pera, and Spadini \cite{fps04}, both from 2004. The setting of \cite{ab04} is for continuous selfmaps on compact polyhedra, which is more general than \cite{fps04}, which requires  differentiable manifolds, but the axioms of \cite{fps04} are simpler. The axioms of \cite{fps04} actually characterize the local fixed point index, a more general invariant which is a localized version of the Lefschetz number. The approach of \cite{ab04} is generalized in \cite{gw06} to the equivariant Lefschetz number and the Reidemeister trace, and the approach of \cite{fps04} is generalized in \cite{stae07a, stae09b, gs12} to coincidence theory and the Reidemeister trace on topological manifolds.

The two papers \cite{ab04} and \cite{fps04} use two separate axiom schema: each of the two approaches use a homotopy invariance axiom, some sort of additivity property, and a property describing specifically the value of the invariant for certain simple spaces and maps. In \cite{ab04}, an additional commutativity axiom is assumed.

The approach in \cite{ab04} directly generalizes a classical characterization of the Euler characteristic by Watts \cite{watt62} in 1962. The Lefschetz number of the identity map is always equal to the Euler characteristic, and so it is natural, when characterizing the Lefschetz number, to attempt generalizations of characterizations for the Euler characteristic. 

In this paper we generalize another characterization of the Euler characteristic in the context of lattice valuations by Hadwiger \cite{hadw57} in the 1950s. Hadwiger showed:
\begin{thm}
Let $X$ be any simplicial complex, and let $C(X)$ be the distributive lattice of all subcomplexes of $X$. Then the Euler characteristic is the unique function $\chi:C(X) \to \R$ satisfying the following two properties:
\begin{enumerate}
\item (Valuation) Let $A,B$ be subcomplexes of some common subdivision $X$. Then $\chi(\emptyset) = 0$, and 
\[ \chi(A\cup B) = \chi(A) + \chi(B) - \chi(A\cap B). \]
\item If $x$ is a simplex, then $\chi(x) = 1$.
\end{enumerate}
\end{thm}
The above is the dimension zero case of Hadwiger's Theorem, which further characterizes all maps on polyconvex bodies in $\R^n$ satisfying the valuation property. See \cite{klai95} for a more recent proof in English. The lattice valuation approach was recently adapted by Pedrini in \cite{pedr12} to give another characterization of the Euler characteristic.

Note that in Hadwiger's theorem and in \cite{fps04}, the invariant is proved to be unique among all real-valued functions obeying the axioms, and it is proved as a consequence of the axioms that it is in fact always an integer. This point of view is carried further in \cite{gs12}, Section 5, where it is assumed only that the values for the index lie in any Abelian group, and then it is proved using the axioms that in fact the values lie in a subgroup isomorphic to (some quotient of) $\Z$. This generality in the values is not present in \cite{ab04}, where it is assumed from the outset that $L(f) \in \Z$. Arkowitz and Brown never seem to use this fact in \cite{ab04}, however, and their proof seems to hold even if we assume only that $L(f) \in \R$.

Our characterization of the Lefschetz number uses essentially the same valuation property as above, but with a more complicated property for determining the value on simplices. In Section \ref{abstractsection} we prove our characterization in the context of simplicial maps on abstract simplicial complexes, and in Section \ref{continuoussection} we generalize to continuous selfmaps of polyhedra, the standard classical setting for the Lefschetz number. In Section \ref{weakensection} we discuss how the typical homotopy invariance axiom can be weakened in our scheme. We show that we need only assume a continuity property for the Lefschetz number, not full homotopy invariance.

The author would like to thank Robert F.\ Brown for many helpful suggestions and conversations.

\section{Simplicial maps on abstract finite simplicial complexes}\label{abstractsection}
Our first characterization of the Lefschetz number is in the setting of maps on abstract simplicial complexes. We begin with some standard definitions.

Let $S$ be a finite set, and $P(S)$ the power set. A \emph{simplicial complex} (or simply \emph{complex}) is a subset $C\subset P(S)$ such that if $x \in C$ and $y\subseteq x$ then $y \in C$. A complex containing exactly one maximal element is a simplex. A one-element subset of $P(S)$ is called an \emph{open simplex}. (Unfortunately, using this terminology, an open simplex is generally not a simplex.) For any $a\in P(S)$, let $\bar a \subset P(S)$ be the simplex whose maximal element is $a$. For a simplex $\bar a$, the \emph{boundary} $\pd \bar a$ is the complex $\bar a - a$.
We will also work with orientations, subdivisions, and chain and homology groups of simplicial complexes, but do not wish to belabor the definitions. For details see \cite{brow71} or \cite{gd03}.

For any complex $X$ we can consider the set $S(X)$ of all subcomplexes $A \subset X$. This is a distributive lattice under union and intersection. A \emph{valuation} $\mu: S(X) \to \R$ is a function satisfying the inclusion-exclusion property $\mu(A \cup B) = \mu(A) + \mu(B) - \mu(A\cap B)$ and $\mu(\emptyset) = 0$. A basic result in the theory of valuations on simplicial complexes is the following, which is found for example as Theorem 3.2.3 of \cite{kr97}:
\begin{thm}\label{simplicesunique}
Any valuation $\mu: S(X) \to \R$ is uniquely determined by its values on simplices $x$ of $X$. The set of these values $\mu(x)$ may be arbitrarily assigned.
\end{thm}

For a particular complex $X \subset P(S)$, let $M(X)$ be the set of all pairs $(f,A)$, where $f:X\to X$ is a simplicial map and $A$ is a subcomplex of some subdivision of $X$. For a pair $(f,A) \in M(X)$ with $A$ a subcomplex of the subdivision $X'$, we have  $f_q:C_q(A) \to C_q(X)$, the induced map on $q$-chains. Let $j_A:C_q(X') \to C_q(A)$ be the linearization of the map which is the identity on simplices of $A$ and zero on simplices outside of $A$. Let $s_{A}: C_q(X) \to C_q(X')$ be the linearization of the map which for a simplex $x$ of $X$ gives value 1 for all simplices arising as subdivisions of $x$, and value 0 otherwise.
We define $f_{A,q}: C_q(X') \to C_q(X')$ as the map $s_{A} \circ f_q \circ j_A$.

For a pair $(f,A)\in M(X)$ and a $q$-simplex $x$ of $A$, let $c(f,x)$ be the coefficient on $x$ in the $q$-chain $f_{A,q}(x)$. This quantity has a natural geometric interpretation: in the case where $x$ is a subdivision of $f(x)$, there is an ``orientation'' which is $+1$ or $-1$ depending on how $x$ ``covers itself'' under $f$, and $c(f,x)$ is precisely this orientation. When $x\not\in A$ or $x$ is not a subdivision of $f(x)$, we will have $c(f,x)=0$. The compositions with $j_A$ and $s_A$ in the definition of $f_{A,q}$ allow us to consider the trace $\tr(f_{A,q})$, which is the sum of all orientations for those $q$-simplices of $A$ which cover themselves.

\begin{thm}\label{abstractunique}
There is a unique function $L:M(X) \to \R$ satisfying the following axioms:
\begin{enumerate}
\item (Valuation axiom) Let $A,B$ be subcomplexes of $X$. Then $L(f,\emptyset)= 0$, and 
\[ L(f,A\cup B) = L(f,A) + L(f,B) - L(f,A\cap B). \]
\item (Simplex axiom) Let $x$ be a simplex. Then
\[ L(f,x) = (-1)^{\dim x} c(f, x) + L(f, \pd x). \]
\end{enumerate}
\end{thm}

\begin{proof}
For the uniqueness, let $L:M(X) \to \R$ satisfy the above axioms, and fix some pair $(f,A)\in M(X)$ such that $A$ is a subcomplex of the subdivision $X'$. We will show that the value of $L(f,A)$ is uniquely determined by the axioms.

Let $\mu:S(X') \to \R$ be given by $\mu(V) = L(f,V)$. Then $\mu$ is a valuation on the lattice $S(X')$ by the valuation axiom for $L$, and so $\mu$ is uniquely determined by its values on the simplices of $X'$ by Theorem \ref{simplicesunique}. But the values $\mu(x) = L(f,x)$ of $\mu$ on simplices $x$ are specified by the simplex axiom. 

Note that the recursive nature of the simplex axiom always terminates in a finite number of steps, since each step decreases the dimension of complexes being considered, and when $\bar x$ is a vertex we will have $L(f,\pd x) = L(f,\emptyset)$, which is zero by the valuation axiom. Thus all values of $\mu$ on simplices are specified by the axioms, and thus $L(f,A) = \mu(A)$ is uniquely determined for all complexes $A$.

For the existence, consider the function defined on open simplices by $L_0(f,\{a\}) = (-1)^{\dim \bar a} c(f,\bar a)$, and extended to $M(X)$ by decomposing a complex into its open simplices:
\[ L_0(f,A) = \sum_{\rho \in A} L_0(f,\{\rho\}). \]
Defined in this way, $L_0$ gives a map $M(X) \to \R$ which satisfies the valuation axiom. We will show that $L_0$ additionally satisfies the simplex axiom. 

Let $\bar a$ be a simplex. Writing $\bar a$ as a union of its elements and using the valuation axiom we have:
\[
L_0(f,\bar a) = \sum_{\rho \subseteq a} L_0(f,\{\rho\}) = L_0(f,\{a\}) + \sum_{\rho \subsetneq a} L_0(f,\{\rho\}) = L_0(f,\{a\}) + L_0(f,\pd \bar a),
\]
where the last equality holds by the valuation axiom because $\pd \bar a$ is the set of all proper subsets of $a$. By the definition of $L_0$, the above reads
\[ L_0(f,\bar a) = (-1)^{\dim \bar a} c(f,\bar a) + L_0(f,\pd \bar a), \]
and so $L_0$ satisfies the simplex axiom as desired.
\end{proof}

We remark that when $f=\id$ is the identity map, our simplex axiom reduces to the condition from Hadwiger's theorem that $\chi(x)$ is always 1: When $x$ is a simplex of dimension $q$, the complex $\pd x$ is topologically a sphere of dimension $q-1$. Assume by induction on the dimension that $L(\id, \pd x) = \chi(\pd x)$, and so $L(\id,\pd x) = \chi(S^{q-1})$ which is $0$ when $q$ is odd and 2 when $q$ is even. Thus our simplex axiom gives the value $(-1) + 2$ when $q$ is odd, and $1 + 0$ when $q$ is even, and in either case we have $L(\id, x) = 1$.

The proof of Theorem \ref{abstractunique} leads to the familiar formula for the Lefschetz number as a trace in homology:
\begin{cor}\label{trace}
The unique function $L: M(X) \to \R$ satisfying the valuation and simplex axioms is computed by the homological formula:
\[ L(f,A) = \sum_{q=0}^{\dim X} (-1)^q \tr(f_{A,q*}:H_q(X') \to H_q(X')), \]
where $X'$ is the subdivision of which $A$ is a subcomplex, and $H_q$ denotes the simplicial homology group in dimension $q$.
\end{cor}
\begin{proof}
By the proof of Theorem \ref{abstractunique}, $L$ can be defined on open simplices and we have $L(f,\{\rho\}) = (-1)^{\dim \bar \rho} c(f,\bar \rho)$, and so, writing $A$ as a union of its elements we obtain:
\begin{align*}
L(f,A) &= \sum_{\rho\in A} L(f,\{\rho\}) = \sum_{\rho \in A} (-1)^{\dim \bar \rho} c(f,\bar \rho) = \sum_{q=0}^{\dim X}  \sum_{\rho \in A, \dim \rho = q}  (-1)^q c(f,\bar \rho) \\
&= \sum_{q=0}^{\dim X} (-1)^q \tr(f_{A,q}:C_q(X) \to C_q(X)).
\end{align*}
By the Hopf Trace Theorem (see \cite{brow71}) the alternating sum of the traces in the chain groups is equal to the alternating sum of the traces in homology, as desired.
\end{proof}

\section{Continuous selfmaps of compact polyhedra}\label{continuoussection}

The uniqueness results for simplicial maps on abstract simplicial complexes extend to the setting of continuous maps on compact polyhedra when we assume an additional homotopy invariance axiom. In this section we will not distinguish between an abstract simplicial complex $X$ and its geometric realization as a compact polyhedron.

By the simplicial approximation theorem, if we have a continuous map $f:X \to X$, there is a subdivision $X'$ of $X$ and a simplicial approximation $f_{X'}:X\to X$ which is the geometric realization of a simplicial map $f':X'\to X$. This subdivision $X'$ and the simplicial approximation $f'$ are not unique, but all these simplicial approximation maps are homotopic to $f$.

Given a compact polyhedron $X$, let $N(X)$ be the set of pairs $(f,A)$, where $A\subset X'$ is a subcomplex of some subdivision $X'$ of $X$ and $f:A\to X$ is a continuous map. We obtain essentially the same uniqueness result in this setting:
\begin{thm}\label{contunique}
There is a unique function $\Lambda:N(X)\to \R$ satisfying the following axioms:
\begin{enumerate}
\item (Valuation axiom) Let $A,B$ be subcomplexes of some common subdivision of $X$. Then $\Lambda(f,\emptyset)=0$, and 
\[ \Lambda(f,A\cup B) = \Lambda(f,A) + \Lambda(f,B) - \Lambda(f,A\cap B). \]
\item (Simplicial map axiom) Let $f$ be a simplicial map, and $x$ be a simplex of some subdivision of $X$. Then
\[ \Lambda(f,x) = (-1)^{\dim x} c(f, x) + \Lambda(f, \pd x), \]
where $c(f,x) \in \{-1,0,1\}$ is defined as in Section \ref{abstractsection}, measuring if and how the simplex $x$ maps onto itself.
\item (Homotopy axiom) Let $f:X\to X$ and $g:X\to X$ be homotopic. Then $\Lambda(f,A) = \Lambda(g,A)$.
\end{enumerate}
\end{thm}
\begin{proof}
Let $N'(X) \subset N(X)$ be the set of pairs $(f,A)$ for which $f$ is simplicial on some subdivision $X'$ of $X$. First we show that there is a unique function $N'(X)\to\R$ satisfying the three axioms above. These elements of $N'(X)$ can be considered as elements of $M(X)$ by simply regarding the abstract simplicial properties of the map and subcomplex. Any function $\lambda:N'(X) \to \R$ which satisfies the axioms of the present theorem will satisfy those of Theorem \ref{abstractunique}, and thus there is a unique function $\lambda:N'(X) \to \R$ satisfying the Valuation and Simplex Axioms. Further, the values of this function $\lambda$ can be computed homologically according to Corollary \ref{trace}, and since the homological trace is homotopy invariant, this $\lambda$ will satisfy the homotopy axiom.

Now we consider the larger set $N(X)$. By the Simplicial Approximation Theorem, for any map $f:X\to X$, there is a subdivision $X'$ of $X$ and a simplicial map $f':X'\to X$ with $f'$ homotopic to $f$. Thus for a pair $(f,A) \in N(X)$, there is a homotopic pair $N(f',A)\in N'(X)$. Any map $\bar\lambda:N(X) \to \R$ satisfying the three axioms must therefore obey $\bar\lambda(f,A) = \lambda(f',A)$, and so the values of $\bar\lambda$ are uniquely determined by the previous paragraph.

For the existence, it remains to check that $\bar\lambda(f,A) = \lambda(f',A)$ is well-defined, that is, that the value $\lambda(f',A)$ does not depend on the choice of the simplicial approximation. This follows from Corollary \ref{trace} which gives a homological formula for $\lambda$ which is homotopy invariant and independent of choices of simplicial structure.
\end{proof}

The normalization axiom above can be weakened if we use the Hopf Construction (see \cite{brow71}) in place of the Simplicial Approximation Theorem. According to the Hopf Construction, any continuous map $f:X\to X$ has a simplicial approximation $f:X' \to X$ for some subdivision $X'$ of $X$ with the property that $f$ has fixed points only inside maximal simplices. We call an simplicial map with all fixed points in maximal simplices a \emph{Hopf simplicial map}.

\begin{thm}\label{hopfunique}
There is a unique function $\Lambda:N(X)\to \R$ satisfying the following axioms:
\begin{enumerate}
\item (Valuation axiom) Let $A,B$ be subcomplexes of some common subdivision of $X$. Then $\Lambda(f,\emptyset)=0$, and 
\[ \Lambda(f,A\cup B) = \Lambda(f,A) + \Lambda(f,B) - \Lambda(f,A\cap B). \]
\item (Hopf simplicial map axiom) Let $f$ be a Hopf simplicial map and let $x$ be a simplex. If $x$ is not a maximal simplex we have $\Lambda(f,x) = 0$, and if $x$ is a maximal simplex we have 
\[ \Lambda(f,x) = (-1)^{\dim X} c(f, x). \]
\item (Homotopy axiom) Let $f:X\to X$ and $g:X\to X$ be homotopic. Then $\Lambda(f,A) = \Lambda(g,A)$.
\end{enumerate}
\end{thm}
\begin{proof}
Since any map is homotopic to a Hopf simplicial map, the proof of Theorem \ref{contunique} will apply exactly, using the Hopf construction in place of the Simplicial Approximation Theorem. We need only show that the formula in the simplicial map axiom of Theorem \ref{contunique} is equivalent to the Hopf simplicial map axiom when $f$ is a Hopf simplicial map. Specifically we must show that when $f$ is Hopf simplicial and $x$ is not a maximal simplex we will have $(-1)^{\dim x} c(f, x) + \Lambda(f, \pd x) = 0$ and when $x$ is a maximal simplex we will have $(-1)^{\dim x} c(f, x) + \Lambda(f, \pd x) = (-1)^{\dim X} c(f, x)$.

When $f$ is Hopf simplicial and $x$ is not a maximal simplex, $f$ will have no fixed points on $x$. Then it is impossible that $x$ is a subdivision of $f(x)$, otherwise there would be a fixed point in $x$ by the Brouwer Fixed Point Theorem. Thus $c(f,x)$ must be zero. Similarly there are no fixed points in the boundary $\pd x$, and by the same argument we will have $c(f,y) = 0$ for any simplex $y$ of $\pd x$. Thus $\Lambda(f,\pd x) = 0$ and so $(-1)^{\dim x} c(f, x) + \Lambda(f, \pd x) = 0$ as desired.

When $x$ is a maximal simplex, still we have no fixed points on the boundary and by the argument above we have $\Lambda(f,\pd x) = 0$. Since $x$ is maximal, $\dim x = \dim X$. Thus $\Lambda(f,x) = (-1)^{\dim X}c(f,x)$ as desired.
\end{proof}

\section{Weakening the homotopy axiom}\label{weakensection}
The proof of Theorem \ref{contunique} suggests that perhaps the homotopy axiom could be weakened or even removed entirely: any function satisfying the axioms of Theorem \ref{abstractunique} must already agree with the homological trace formula of Corollary \ref{trace}, which itself is already homotopy invariant. We cannot simply discard the homotopy axiom however, because not all maps of $N(X)$ will be simplicial on some subdivision of $X$. For such a map the [Hopf] simplicial map axiom simply does not apply, and so the valuation and [Hopf] simplicial map axioms will not suffice to uniquely define the value of $\Lambda$. 

Nonetheless the usage of the homotopy axiom in the proof of Theorems \ref{contunique} and \ref{hopfunique} is still fairly specific, the axiom only being invoked when passing from a map to its [Hopf] simplicial approximation. These approximations can be accomplished using arbitrarily small homotopies.

In fact, the set of all Hopf simplicial maps forms a dense subset of the space $X^X$ of all selfmaps of $X$, when this space is given a natural topology in which homotopies of maps correspond to paths in $X^X$ (for example the compact-open topology). Since the first two axioms used in Theorem \ref{hopfunique} suffice to define $\Lambda$ uniquely on the dense set of Hopf simplicial maps, we need only assume continuity of $\Lambda$ in $f\in X^X$ to obtain a unique invariant on all maps. Thus we can weaken the homotopy axiom to a continuity axiom.

\begin{thm}\label{continuityaxiom}
There is a unique function $\Lambda:N(X)\to \R$ satisfying the following axioms:
\begin{enumerate}
\item (Valuation axiom) Let $A,B$ be subcomplexes of some common subdivision of $X$. Then $\Lambda(f,\emptyset)=0$, and 
\[ \Lambda(f,A\cup B) = \Lambda(f,A) + \Lambda(f,B) - \Lambda(f,A\cap B). \]
\item (Hopf normalization axiom) Let $f$ be a Hopf simplicial map on a subdivision $X'$, and $x$ be a simplex of $X'$. If $x$ is not a maximal simplex we have $\Lambda(f,x) = 0$, and if $x$ is a maximal simplex we have 
\[ \Lambda(f,x) = (-1)^{\dim X} c(f, x). \]
\item (Continuity axiom) $\Lambda(f,A)$ depends continuously on $f\in X^X$. 
\end{enumerate}
\end{thm}

For a real-valued function $\Lambda$, this continuity axiom is considerably weaker than homotopy invariance, which requires that $\Lambda(f, A)$ is constant on path components of $X^X$. (If $\Lambda$ is assumed from the outset to be integer valued, this distinction vanishes.) Ours is the first axiom scheme for the Lefschetz number which weakens the homotopy invariance assumption. Because the arguments used in \cite{fps04} are generally based on transversality, we conjecture that the homotopy axiom in that setting (and that of \cite{gs12}) could be similarly weakened. This may be a bit more difficult in the approach of \cite{ab04}, where ``large'' homotopies are used.

We note in conclusion that it should be straightforward to generalize this paper's result to a characterization of the Reidemeister trace as in \cite{stae09b} by replacing $c(f,x)$ throughout by an appropriate twisted conjugacy class with sign given by $c(f,x)$. Generalizing to coincidence theory on orientable spaces seems nontrivial, however, since it is not clear what to use as a substitute for $c(f,x)$. The Lefschetz coincidence number involves Poincar\'e duality, which is nontrivial to mimic with a pair of simplicial maps. 

\bibliographystyle{hplain}

\end{document}